\documentclass[12pt]{article}
\usepackage{graphicx}
\usepackage{amssymb,latexsym,amscd,amsmath,amsfonts,amsthm,pb-diagram,enumerate}
\usepackage[mathscr]{eucal}
\numberwithin{equation}{section}
\newtheorem{theorem}{Theorem}[section]
\newtheorem{proposition}[theorem]{Proposition}
\newtheorem{corollary}[theorem]{Corollary}
\newtheorem{lemma}[theorem]{Lemma}
\newtheorem{question}[theorem]{Question}
\newtheorem{remark}[theorem]{Remark}

\newtheorem{definition}[theorem]{Definition}
\begin{document}
\begin{center} {\Large A study of the length function of generalized fractions\\ of modules}
 \footnote [1]{
\indent{\bf AMS Classification 2010:}13H15; 13D40; 13D45.\\
This research is supported by Vietnam National Foundation for Science
and Technology Development (NAFOSTED).}
\end{center}
\begin{center}
                {{\sc Marcel
Morales}\\
{\small Universit\'e de Grenoble I, Institut Fourier,
UMR 5582, B.P.74,\\
38402 Saint-Martin D'H\`eres Cedex,\\
and ESPE, Universit\'e Lyon 1, 5 rue Anselme,\\ 69317 Lyon Cedex (FRANCE)}\\
 E-mail address: morales@ujf-grenoble.fr\\
 }
 {{\sc Pham Hung Quy }\\
{\small Department of Mathematics, FPT University, \\
8 Ton That Thuyet, Hanoi, Vietnam\\ }
E-mail address: quyph@fpt.edu.vn
}
\end{center}
\begin{abstract}
Let $(R, \frak m)$ be a Noetherian local ring and $M$ a finitely generated $R$-module of dimension $d$. Let $\underline{x} = x_1, ..., x_d$ be a
system of parameters of $M$ and $\underline{n} = (n_1, ..., n_d)$ a $d$-tuple of positive integers. In this paper we study the length of generalized
 fractions $M (1/(x_1, ..., x_d, 1))$ which was introduced  by Sharp and Hamieh in \cite{ShH85}.
First, we study the growth of the function
$J_{\underline{x}, M}(\underline{n}) = \ell(M (1/(x_1^{n_1}, ..., x_d^{n_d}, 1))) - n_1...n_d e(\underline{x};M)$.
Then we give an explicit calculation for the function $J_{\underline{x}, M}(\underline{n})$ in the case where $M$
admits a Macaulayfication. Most previous results on this topic are now easy to understand and to improve.
\end{abstract}
\section{Introduction}
Throughout this paper, let $(R, \frak m)$ be a Noetherian local ring and $M$ a finitely generated $R$-module of dimension $d$.
Let $\underline{x} = x_1, ..., x_d$ be a system of parameters of $M$.
In this paper we study the length of generalized
 fractions $M (1/(x_1, ..., x_d, 1))$ which was introduced  by Sharp and Hamieh in \cite{ShH85}.
 It has been proved in \cite[Lemma 2.3]{CK99} that $M/((\underline{x})_M^{\lim})$ is isomorphic to
$M(1/(x_1, ..., x_d, 1))$, where $$(\underline{x})_M^{\lim} = \bigcup_{n>0}\big((x_1^{n+1}, ..., x_d^{n+1})M: (x_1...x_d)^n \big).$$  We call
$(\underline{x})_M^{\lim} $ the {\it limit closure}. If $M = R$ we write $(\underline{x})^{\lim}$.

It should be noted that the Hochster monomial conjecture is equivalent to the claim $(\underline{x})^{\lim} \neq R$ for all system of parameters
 $\underline{x}$.

Let $\underline{n} = (n_1,...,n_d)$ be a $d$-tuple
of positive integers and $\underline{x}^{\underline{n}} =
x_1^{n_1},...,x_d^{n_d}$. We consider the functions  in $\underline{n}$,
$$I_{\underline{x}, M}(\underline{n}) = \ell(M/(\underline{x}^{\underline{n}})M) -
e(\underline{x}^{\underline{n}};M),$$
$$J_{\underline{x}, M}(\underline{n}) = e(\underline{x}^{\underline{n}},M) -
\ell(M/(\underline{x}^{\underline{n}})_M^{\lim}),$$
where $e(\underline{x};M)$ is the Serre multiplicity of $M$ with
respect to the sequence $\underline{x}$. In several papers N.T. Cuong et als, showed that the least degree of all polynomials in
$\underline{n}$ bounding above $I_{\underline{x}, M}(\underline{n})$
is independent of the choice of $\underline{x}$.
It is called {\it the
polynomial type } of $M$, and denoted by $p(M)$.
The behavior of the function  $ J_{\underline{x}, M}(\underline{n})$ was studied in \cite{MN03} and \cite{CMN03}.
In general $J_{\underline{x}, M}(\underline{n})$ is not a polynomial in $\underline{n}$.
Furthermore, the least degree of polynomials bounding above $J_{\underline{x}, M}(\underline{n})$ is independent of the choice of $\underline{x}$.
(see \cite[Theorem 4.4]{CHL99}).
It is called {\it the
polynomial type of generalized fractions} of $M$, and denoted by $pf(M)$.

These two functions are closely related. In general it was proved in \cite[Theorem 4.5]{M95} that
$pf(M) \le p(M)$. Our first result proves that
 if $M$ is unmixed and $\underline{x}$ is a certain system of parameters, then
 $ I_{\underline{x}, M}(\underline{n})\leq 2^{d-2}J_{\underline{x}, M}(\underline{n}), $ which implies that $pf(M)=p(M).$

 Our second result consist to study the function $J_{\underline{x}, M}(\underline{n})$ in the case where $M$ admits a Macaulayfication and we can express
$J_{\underline{x}, M}(\underline{n})$ in terms
of the Non Cohen-Macaulay locus of $M$.  As an application, in characteristic $p>0$, we establish a
connection between $J_{\underline{x}, M}(\underline{n})$
 and the Hilbert-Kunz function, and prove by using a recent result of Brenner \cite{B14}, the existence of a local ring and a system of parameters such that the function
 $J_{\underline{x}, M}(\underline{n})$, with $n=n_1=...=n_d$, cant be defined by a finite set of polynomials.

\section{Preliminaries}
First we recall the notion of {\it polynomial type} of a module.
Let $(R, \frak m)$ be a Noetherian local ring, $M$ a finitely
generated $R$-module of dimension $d$, $\underline{x} =
x_1,...,x_d$ a system of parameters of $M$, and $\underline{n} = (n_1,...,n_d)$ a $d$-tuple
of positive integers. We set $\underline{x}^{\underline{n}} =
x_1^{n_1},...,x_d^{n_d}$ and we consider the function in $\underline{n}$
$$I_{\underline{x}, M}(\underline{n}) = \ell(M/(\underline{x}^{\underline{n}})M) -
e(\underline{x}^{\underline{n}};M),$$
where $e(\underline{x};M)$ is the Serre multiplicity of $M$ with
respect to the sequence $\underline{x}$. N.T. Cuong in \cite[Theorem
2.3]{C92} showed that the least degree of all polynomials in
$\underline{n}$ bounding above $I_{\underline{x}, M}(\underline{n})$
is independent of the choice of $\underline{x}$.
\begin{definition}\rm
The least degree of all polynomials in $\underline{n}$ bounding
above $I_{\underline{x}, M}(\underline{n})$ is called {\it the
polynomial type} of $M$, and is denoted by $p(M)$.
\end{definition}
The following basic properties of $p(M)$ can be found in \cite{C92}.
\begin{remark}\rm
\begin{enumerate}[{(i)}]
\item We have $p(M) = p(\widehat{M})  \leq d-1$, where $\widehat{M}$ is the $\frak m$-adic completion of $M$.
\item  An $R$-module $M$ is Cohen-Macaulay if and only if $p(M) = -\infty$.
Moreover, $M$ is generalized Cohen-Macaulay if and only if $p(M) \le
0$.
\end{enumerate}
\end{remark}
Let $\frak a_i(M) = \mathrm{Ann}H^i_{\frak m}(M)$ for $0 \leq i \leq d-1$
and $\frak a(M) = \frak a_0(M) \cdots \frak a_{d-1}(M)$. We denote
by $NC(M)$ the non-Cohen-Macaulay locus of $M$ i.e. $NC(M) = \{\frak
p \in \mathrm{supp}(M)\,|\, M_{\frak p} \,\, \text{is not
Cohen-Macaulay} \}$. Recall that $M$ is called {\it
equidimensional} if $\dim M = \dim R/\frak p$ for all minimal
associated primes of $M$. The polynomial type of a module can be well understood by the annihilator of local cohomology as follows.
\begin{proposition}[\cite{C91}, Theorem 1.2] \label{P2.3} Suppose that $R$ admits a dualizing complex. Then
\begin{enumerate}[{(i)}]\rm
\item $p(M) = \dim R/\frak a(M)$.
\item {\it If $M$ is equidimensional then $p(M) = \dim (NC(M))$.}
\end{enumerate}
\end{proposition}
Although the function $I_{\underline{x}, M}(\underline{n})$ is not a polynomial in general, it has a good behavior for some special systems of parameters.
\begin{definition}[\cite{C95}] \rm  A system of parameters $x_1,...,x_d$ of $M$ is called {\it $p$-standard} if $x_d \in \frak
a(M)$ and $x_i \in \frak a(M/(x_{i+1},...,x_d)M)$ for all $i =
d-1,...,1$.
\end{definition}
\begin{definition}[\cite{Hu82}, \cite{GY86}]\rm
\begin{enumerate}[{(i)}]
\item A sequence in $R$, $\underline{x} = x_1,...,x_s$ is called a {\it
$d$-sequence} of $M$ if $(x_1,...,x_{i-1})M:x_j =
(x_1,...,x_{i-1})M:x_ix_j$ for all $i \leq j \leq s$.
\item A sequence $\underline{x} = x_1,...,x_s$ is called a {\it
strong $d$-sequence} if $\underline{x}^{\underline{n}} =
x_1^{n_1},...,x_s^{n_s}$ is a $d$-sequence for all $\underline{n} =
(n_1,...,n_s) \in \mathbb{N}^s$.
\end{enumerate}
\end{definition}

For important properties of $d$-sequence, see
\cite{Hu82} and \cite{Tr83}.

\begin{definition}[\cite{CC07-1}]\rm
 A sequence of elements $\underline{x} = x_1,...,x_s$ is called a {\it
$dd$-sequence} of $M$ if $\underline{x}$ is a strong $d$-sequence of
$M$ and the following conditions are satisfied:
\begin{enumerate}[{(i)}]
\item $s=1$ or,
\item $s>1$ and $\underline{x}' = x_1,...,x_{s-1}$ is a $dd$-sequence of $M/x_s^n$ for all $n \geq 1$.
\end{enumerate}
\end{definition}
The function $I_{\underline{x}, M}(\underline{n})$ is a polynomial for a $p$-standard system of parameters or $dd$-sequence of parameters (see \cite[Theorem 2.6 (ii)]{C95} and \cite[Theorem 1.2]{CC07-1}).
\begin{proposition}\label{P2.7}
A system of parameters $\underline{x} = x_1,...,x_d$ of $M$ is a $dd$-sequence iff for all $n_1,...,n_d>0$ we have
$$I_{\underline{x}, M}(\underline{n}) =
\sum_{i=0}^{p(M)} n_1...n_i e_i,$$ where
$e_i = e(x_1,...,x_i; 0:_{M/(x_{i+2},...,x_d)M}x_{i+1})$ and $e_0 =
\ell(0:_{M/(x_{2},...,x_d)M}x_{1})$. Moreover a $p$-standard system of parameters is a $dd$-sequence system of parameters.
\end{proposition}
In order to introduce the notion of {\it polynomial type of generalized fractions} we recall the notion of {\it limit closure} of a parameter ideal.
\begin{definition} \rm Let $\underline{x} = x_1, ..., x_d$ be a system of parameters of $M$. Then the {\it limit closure} of $\underline{x}$ in $M$ is a submodule of $M$ defined by
$$(\underline{x})_M^{\lim} = \bigcup_{n>0}\big((x_1^{n+1}, ..., x_d^{n+1})M: (x_1...x_d)^n \big),$$
when $M = R$ we write $(\underline{x})^{\lim}$ for short.
\end{definition}
For a study of limit closure we refer to \cite{CQ15}.
\begin{remark}\label{R2.9}
  \begin{enumerate}[{(i)}]\rm
\item It is well known that $(\underline{x})M = (\underline{x})_M^{\lim}$ if and only if $\underline{x}$ is an $M$-sequence i.e. $M$ is Cohen-Macaulay.
\item The quotient $(\underline{x})_M^{\lim}/(\underline{x})M$ is the kernel of the canonical map
$$H^d(\underline{x};M) \to H^d_{\frak m}(M).$$
\item (see \cite[Lemma 2.4]{CC15}) If $x_1, ..., x_d$ is a $dd$-sequence we have
$$(\underline{x})_M^{\lim} = \sum_{i=1}^d\big[ (x_1, ..., \widehat{x_i}, ..., x_d)M :_Mx_i  \big]+ (\underline{x})M.$$
\end{enumerate}
\end{remark}
Similarly to the notion of polynomial type, we consider the function in $\underline{n}$
$$J_{\underline{x}, M}(\underline{n}) = e(\underline{x}^{\underline{n}},M) -
\ell(M/(\underline{x}^{\underline{n}})_M^{\lim}).$$
In general $J_{\underline{x}, M}(\underline{n})$ is not a polynomial in $\underline{n}$ (cf. \cite{CMN03}) but it is bounded by polynomials. Furthermore, the least degree of polynomials bounding above $J_{\underline{x}, M}(\underline{n})$ is independent of the choice of $\underline{x}$ (see \cite[Theorem 4.4]{CHL99}).
\begin{definition}\rm
The least degree of all polynomials in $\underline{n}$ bounding
above $J_{\underline{x}, M}(\underline{n})$ is called {\it the
polynomial type of generalized fractions} of $M$, and denoted by $pf(M)$.
\end{definition}
Now we recall the notion of {\it unmixed component} of $M$ which is closely related with the limit closure and the polynomial type of generalized fractions.
\begin{definition} \rm The largest submodule of $M$ of dimension less than $d$ is called {\it the unmixed component} of $M$ and it is denoted by $U_M(0)$.
\end{definition}
It should be noted that if $\cap_{\frak p \in \mathrm{Ass}M}N(\frak p) = 0_M$ is a reduced primary decomposition of the zero submodule of $M$, then
$U_M(0) = \cap_{\frak p \in \mathrm{Assh}M}N(\frak p)$, where $\mathrm{Assh}M = \{\frak p \in \mathrm{Ass}M | \dim R/\frak p = \dim M\}$.
\begin{remark} \label{R2.12} \rm
   \begin{enumerate}[{(i)}]\rm
\item In \cite[Theorem 4.1]{CQ15} it is proved that $U_M(0) = \cap_{n}(\underline{x}^{[n]})^{\lim}_M$ for any system of parameters $\underline{x}$ of $M$, where we denote $\underline{x}^{[n]} = x_1^n, ..., x_d^n$.
\item (cf. \cite[Theorem 3.1]{CN03}) Suppose that $R$ admits a dualizing complex then $pf(M) = -\infty$ (resp. $pf(M) \le 0$) if and only if $M/U_M(0)$ is Cohen-Macaulay (resp. generalized Cohen-Macaulay).
\end{enumerate}
\end{remark}
Recently, N.T. Cuong and the second author study the splitting of local cohomology (cf. \cite{CQ11}, \cite{CQ16}), this will provide the main tool for the proof of our first result in this paper. We collect here some results which we need in the sequel. Set
$$\mathfrak{b}(M) = \bigcap_{\underline{x};i=1}^d
\mathrm{Ann}(0:x_i)_{M/(x_1,...,x_{i-1})M},$$
where $\underline{x} = x_1, ..., x_d$ runs over all systems of parameters of $M$. By \cite[Satz 2.4.5]{Sch82} we have
$$\mathfrak{a}(M) \subseteq \mathfrak{b}(M) \subseteq \mathfrak{a}_0(M) \cap \cdots \cap \mathfrak{a}_{d-1}(M).$$

We have the following splitting property.
\begin{theorem}[\cite{CQ16}, Corollary 3.5] \label{T2.13}
  Let $x \in \mathfrak{b}(M)^3$ be a parameter element of $M$. Let $U_M(0)$ be the unmixed component of $M$ and set $\overline{M} = M/U_M(0)$.
 Then
$$H^i_{\mathfrak{m}}(M/xM) \cong H^i_{\mathfrak{m}}(M) \oplus H^{i+1}_{\mathfrak{m}}(\overline{M})$$
for all $i<d-1$.
\end{theorem}

\begin{lemma}\label{L2.14}
  Let $N \subseteq H^0_{\frak m}(M)$ be a submodule of finite length. Then $\frak b(M) \subseteq \frak b(M/N)$.
\end{lemma}
\begin{proof} Let $x_1,...,x_d$ be an arbitrary system of parameters  of $M/N$. It is also a system of parameters of $M$. By definition of $\frak b(M/N)$, we need only to prove that
$$\frak{b}(M) \subseteq \mathrm{Ann}\frac{[(x_1,...,x_{i-1})M+N]:x_i}{(x_1,...,x_{i-1})M+N}$$
for all $i \le d$. Choose a positive integer $n_0$ such that $x_i^{n_0}N=0$ and for all $n\geq n_0$ we have
$$(x_1,...,x_{i-1})M:x_i^n=(x_1,...,x_{i-1})M:x_i^{n_0},$$
$$[(x_1,...,x_{i-1})M+N]:x_i^n=[(x_1,...,x_{i-1})M+N]:x_i^{n_0}.$$
So
$$[(x_1,...,x_{i-1})M+N]:x_i^{n_0} \subseteq  (x_1,...,x_{i-1})M:x_i^{2n_0} \subseteq [(x_1,...,x_{i-1})M+N]:x_i^{2n_0}.$$
Hence $(x_1,...,x_{i-1})M:x_i^{2n_0} = [(x_1,...,x_{i-1})M+N]:x_i^{2n_0}$ and we have
\begin{eqnarray*}
\mathrm{Ann}\frac{[(x_1,...,x_{i-1})M+N]:x_i}{(x_1,...,x_{i-1})M+N}
&\supseteq&
\mathrm{Ann}\frac{[(x_1,...,x_{i-1})M+N]:x_i^{2n_0}}{(x_1,...,x_{i-1})M+N}\\
&=&
\mathrm{Ann}\frac{(x_1,...,x_{i-1})M:x_i^{2n_0}}{(x_1,...,x_{i-1})M+N}\\
&\supseteq&
\mathrm{Ann}\frac{(x_1,...,x_{i-1})M:x_i^{2n_0}}{(x_1,...,x_{i-1})M}\\
&\supseteq& \frak{b}(M).
\end{eqnarray*}
\end{proof}

The following notion of system of parameters is closed related with $p$-standard and $dd$-sequence system of parameters and very useful in this paper.

\begin{definition}\rm A system of parameters $x_1, ..., x_d$ is called a $C$-system of parameters of $M$ if $x_d \in \frak b(M)^3$ and $x_i \in \frak b(M/(x_{i+1}, ..., x_d)M)^3$ for all $i = d-1, ..., 1$.
\end{definition}
We call $C$-system of parameters in honor of Professor N.T. Cuong. If $(R, \frak m)$ is the quotient of a Cohen-Macaulay ring then we always have that $\dim R/\frak a(M) < \dim M$ for every finitely generated $R$-module $M$. So every finitely generated $R$-module $M$ admits a $C$-system of parameters.

\begin{lemma}\label{L2.15} Let $x_1, ..., x_d$ be a $C$-system of parameters of $M$. Then
   \begin{enumerate}[{(i)}]\rm
\item $x_1, ..., x_d$ is a $dd$-sequence.
\item $x_1^{n_1}, ..., x_d^{n_d}$ is a $C$-system of parameters of $M$ for all $n_1, ...., n_d \ge 1$.
\item For all $i \le d$ we have $x_1, ..., x_{i-1}, x_{i+1}, ..., x_d$ is a $C$-system of parameters of $M/x_iM$.
\item Let $N \subseteq H^0_{\frak m}(M)$ be a submodule of finite length. Then $x_1, ..., x_d$ is a $C$-system of parameters of $M/N$.
\end{enumerate}
\end{lemma}
\begin{proof} (i) is \cite[Proposition 4.6]{CQ16}, (ii) is \cite[Corollary 4.5]{CQ16} and (iii) is \cite[Lemma 2.10]{CQ16}.\\
(iv) For each $i \le d$ we have $ M/((x_{i+1}, ...,x_d)M + N)$ is a quotient module of $M/(x_{i+1}, ...,x_d)M$ by a submodule of finite length. So $\frak b(M/(x_{i+1}, ...,x_d)M) \subseteq \frak b(M/((x_{i+1}, ...,x_d)M + N))$ by Lemma \ref{L2.14}. Thus
$$x_i \in \frak b(M/(x_{i+1}, ...,x_d)M)^3 \subseteq \frak b(M/((x_{i+1}, ...,x_d)M + N))^3.$$
\end{proof}

\section{On the polynomial type of generalized fractions}
Since $p(M)$ and $pf(M)$ do not change after passing to the completion. In this section we assume that $(R, \frak m)$ is the image of a Cohen-Macaulay local ring. For each system of parameters $\underline{x} = x_1, ..., x_d$ set
$$I_{\underline{x},M} = \ell(M/(\underline{x})M) - e(\underline{x}; M)$$
and
$$J_{\underline{x},M} = e(\underline{x}; M) -  \ell(M/(\underline{x})_M^{\lim}).$$
It should be noted that $I_{\underline{x},M}$ is much easier to understand than $J_{\underline{x},M}$.
\begin{lemma}\label{L3.1} Let $M$ be a generalized Cohen-Macaulay module and $\underline{x} = x_1, ..., x_d$ a standard system of parameters of $M$. Then
  \begin{enumerate}[{(i)}]\rm
\item $I_{\underline{x},M} = \sum_{i=0}^{d-1}\binom{d-1}{i}\ell(H^i_{\frak m}(M)).$
\item $J_{\underline{x},M} = \sum_{i=1}^{d-1}\binom{d-1}{i-1}\ell(H^i_{\frak m}(M))$.
\end{enumerate}
\end{lemma}
\begin{proof} For the definition of standard system of parameters  and the proof of (i) see \cite{Tr86}, (ii) follows from \cite[Theorem 5.1]{CHL99}.
\end{proof}

\begin{lemma}\label{L3.2} Let $\underline{x} = x_1, ..., x_d$ be a system of parameters of $M$ and $U_M(0)$ the unmixed component of $M$. Set $\overline{M} = M/U_M(0)$ we have
\begin{enumerate}[{(i)}]\rm
\item $J_{\underline{x},M} = J_{\underline{x},\overline{M}}$.
\item $J_{\underline{x}, M}(\underline{n}) = J_{\underline{x}, \overline{M}}(\underline{n})$ for all $\underline{n}$.
\item $pf(M) = pf(\overline{M})$.
\end{enumerate}
\end{lemma}
\begin{proof} (i) Since $\dim U_M(0) < d$ we have $e(\underline{x}; M) = e(\underline{x}; \overline{M})$. For each $n\ge 1$ we set $\underline{x}^{[n]} = x_1^n, ..., x_d^n$. By Remark \ref{R2.12} we have $U_M(0) = \cap_{n \ge 1}(\underline{x}^{[n]})_M^{\lim}$. By \cite[Proposition 2.6]{CQ15} we have
  $$\ell(M/(\underline{x})_M^{\lim}) = \ell(\overline{M}/(\underline{x})_{\overline{M}}^{\lim}).$$
  Therefore $J_{\underline{x},M} = J_{\underline{x}, \overline{M}}$.\\
(ii) follows from (i) and (iii) follows from (ii).
\end{proof}
By the above lemma, we can assume that $M$ is unmixed i.e. $U_M(0) = 0$, for the computation of either the function $J_{\underline{x}, M}(\underline{n})$ or $pf(M)$. The following is important for our inductive technique.
\begin{remark}\label{R3.3} \rm Let $M$ be an unmixed finitely generated $R$-module of dimension $d$. Then
 \begin{enumerate}[{(i)}]
\item $H^1_{\frak m}(M)$ is finitely generated provided $d \ge 2$  (for example see \cite[Lemma 3.1]{GN01}).
\item The set $$\mathcal{F}(M) = \{\mathfrak{p}\in \mathrm{Spec}(R)\,|\, \dim M_{\mathfrak{p}}>
1=\mathrm{depth}M_{\mathfrak{p}},\, \mathfrak{p} \neq \mathfrak{m}
\}$$ is finite (cf. \cite[Lemma 3.2]{GN01}).
\item Let $\underline{x} = x_1, ..., x_d$ be a $C$-system of parameters of $M$. Then
$$\mathcal{F}(M) = \mathrm{Ass}U_{M/x_dM}(0) \setminus \{\frak m\}$$
and $x_1 \notin \frak p$ for all $\frak p \in \mathcal{F}(M)$. Hence $\mathrm{Ass}M/x_1M \subseteq \mathrm{Assh} M/x_1M \cup \{\frak m\}$, so $U_{M/x_1M}(0) \cong H^0_{\frak m}(M/x_1M)$ (cf. \cite[Proposition 4.11, Remark 4.12]{CQ16}).
\end{enumerate}
 \end{remark}

\begin{lemma}\label{L3.4}
Let $M$ be an unmixed finitely generated $R$-module of dimension $d\geq 2$ and $\underline{x} =
x_1, ...,x_d$ a $C$-system of parameters of $M$. Then
$x_1.H^{1}_{\mathfrak{m}}(M) = 0$ and
$\ell(H^{1}_{\mathfrak{m}}(M)) \leq I_{\underline{x}, M}$.
\end{lemma}
\begin{proof} Set $M_d = M/x_dM$. Since $M$ is unmixed, by Theorem \ref{T2.13} we have $H^{1}_{\frak{m}}(M) \cong H^{0}_{\frak{m}}(M_d)$. By Lemma \ref{L2.15} we have
$\underline{x}' = x_1, ...,x_{d-1}$ is a $dd$-sequence of $M_d$ so $H^{0}_{\frak{m}}(M_d) = 0:_{M_d}x_1$. Hence
$x_1.H^{1}_{\frak{m}}(M) = 0$. Moreover the properties of $dd$-sequences imply that $H^{0}_{\frak{m}}(M_d) \cap (\underline{x}')M_d = 0$. Thus
\begin{eqnarray*}
\ell (M_d/(\underline{x}')M_d) = \ell(H^{0}_{\frak{m}}(M_d)) + \ell (\overline{M_d}/(\underline{x}')\overline{M_d}) &\ge& \ell(H^{1}_{\frak{m}}(M)) + e(\underline{x}'; \overline{M_d})\\
&=& \ell(H^{1}_{\frak{m}}(M)) + e(\underline{x}'; M_d),
\end{eqnarray*}
where $\overline{M_d} = M_d/H^{0}_{\frak{m}}(M_d)$. Therefore
$$\ell(H^{1}_{\mathfrak{m}}(M)) \le I_{\underline{x}', M_d} = I_{\underline{x}, M}.$$
For the last equality notice that since $x_d$ is $M$-regular we have $e(\underline{x}; M) = e(\underline{x}'; M_d)$. The proof is complete.
\end{proof}

\begin{lemma}\label{L3.5}
Let $M$ be an unmixed finitely generated $R$-module of dimension $d\geq 3$ and $\underline{x} =
x_1, ..., x_d$ a $C$-system of parameters of $M$. Set $M_1 =
M/x_1M$ and $\underline{x}' = x_2, ..., x_d$ we have
$I_{\underline{x},M} \le 2I_{\underline{x}', \overline{M}_1}$, where $\overline{M_1} = M_1/H^0_{\frak m}(M_1)$.
\end{lemma}
\begin{proof} Since $x_1$ is $M$-regular we have $e(\underline{x}; M) = e(\underline{x}'; M_d)$. So $I_{\underline{x},M} = I_{\underline{x}', M_1}$.
By Lemma \ref{L2.15} we have $\underline{x}' = x_2, ..., x_d$ is a $C$-system of parameters of $M_1$. Similar to the proof of the previous result we have
$$I_{\underline{x}',M_1} = I_{\underline{x}', \overline{M}_1} + \ell(H^{0}_{\frak{m}}(M_1)).$$
Thus we need only to prove that $\ell(H^{0}_{\frak{m}}(M_1)) \le
I_{\underline{x}', \overline{M}_1}$. Consider the following short exact sequence
$$0 \longrightarrow M \overset{x_1\cdot}{\longrightarrow} M \longrightarrow M_1 \longrightarrow 0.$$
By Lemma \ref{L3.4} we have $x_1.H^{1}_{\frak{m}}(M) = 0$. So by applying the local cohomology functor to the above short exact sequence we have
$ H^{0}_{\mathfrak{m}}(M_1)\cong H^{1}_{\mathfrak{m}}(M)$ and
$$0 \longrightarrow H^{1}_{\mathfrak{m}}(M) \longrightarrow H^{1}_{\mathfrak{m}}(M_1).$$
Thus $$\ell(H^{0}_{\mathfrak{m}}(M_1)) =
\ell(H^{1}_{\mathfrak{m}}(M)) \leq
\ell(H^{1}_{\mathfrak{m}}(M_1)).$$ On the other hand by Remark \ref{R3.3} we have $\overline{M_1}$ is unmixed, and $\underline{x}'$ is a $C$-system of parameters of $\overline{M_1}$ by Lemma \ref{L2.15}. So
$$ \ell(H^{1}_{\mathfrak{m}}(M_1)) =
\ell(H^{1}_{\mathfrak{m}}(\overline{M}_1))\leq
I_{\underline{x}', \overline{M}_1}$$ by Lemma \ref{L3.4}. Thus $\ell(H^{0}_{\mathfrak{m}}(M_1)) \leq
I_{\underline{x}', \overline{M}_1}$. The proof is complete.
\end{proof}

\begin{proposition}\label{P3.6}
Let $M$ be an unmixed finitely generated $R$-module of dimension $d$ and $\underline{x} =
x_1, ..., x_d$ a $C$-system of parameters of $M$. Then
$I_{\underline{x},M} \leq 2^{d-2}J_{\underline{x},M}$.
\end{proposition}
\begin{proof}
We proceed by induction on $d$. The case $d=1$ is trivial since $M$ is Cohen-Macaulay. For $d=2$ by Lemma \ref{L3.1} we have
$$I_{\underline{x},M} = \ell(H^{1}_{\mathfrak{m}}(M)) = J_{\underline{x},M}.$$
Assume that $d \geq 3$ and the assertion was proved for
$d-1$. Set $M_1 = M/x_1M$ and $\underline{x}' = x_2, ...,x_d$ we have
\begin{eqnarray*}
  I_{\underline{x},M} & \le & 2I_{\underline{x}', \overline{M}_1} \quad \quad \,\,(\text{By Lemma \ref{L3.5}})\\
& \leq & 2^{d-2}J_{\underline{x}',\overline{M}_1} \quad (\text{By induction})\\
& = & 2^{d-2}J_{\underline{x}',M_1} \quad \,(\text{By Lemma \ref{L3.2}}).
\end{eqnarray*}
Since $x_1$ is $M$-regular we have $e(\underline{x};M) = e(\underline{x}';M_1)$. On the other hand we have
$$(\underline{x}')_{M_1}^{\lim} = \bigcup_n[(x_1, x_2^{n+1}, ..., x_d^{n+1})M:_M(x_2, ...,x_d)^n]/x_1M \subseteq (\underline{x})_{M}^{\lim}/x_1M.$$
So $\ell(M/(\underline{x})_{M}^{\lim}) \leq \ell(M_1/(\underline{x}')_{M_1}^{\lim})$. Thus $J_{\underline{x}',M_1} \le J_{\underline{x}, M}$.  Therefore we get the assertion $I_{\underline{x},M} \leq 2^{d-2}J_{\underline{x},M}$.
\end{proof}

\begin{theorem}\label{T3.7} Let $(R, \frak m)$ be the image of a Cohen-Macaulay local ring and $M$ an unmixed finitely generated $R$-module of dimension $d$. Then $pf(M) = p(M)$. Moreover $pf(M) = \dim R/\frak a(M)$.
\end{theorem}
\begin{proof} By \cite[Theorem 4.5]{M95} we have $pf(M) \le p(M)$. Thus we need only to prove $pf(M) \ge p(M)$.
Let $\underline{x} = x_1, ...,x_d$ be a $C$-system of parameters of $M$. By Lemma \ref{L2.15}, for all $d$-tuples of positive integers $\underline{n} = (n_1,...,n_d)$ we have
$\underline{x}^{\underline{n}} = x_1^{n_1},...,x_d^{n_d}$ is also a $C$-system of parameters. By Proposition \ref{P3.6} we have
$$I_{\underline{x}, M}(\underline{n}) = I_{\underline{x}^{\underline{n}},M} \leq 2^{d-2}J_{\underline{x}^{\underline{n}}, M} = 2^{d-2}J_{\underline{x}, M}(\underline{n})$$
for all $\underline{n} = (n_1,...,n_d) \in \mathbb{N}^d$. Thus $p(M) \leq pf(M)$. The last assertion follows from Proposition \ref{P2.3}. The proof is complete.
\end{proof}
The next result is a consequence of the above Theorem and Lemma \ref{L3.2}.
\begin{corollary}\label{C3.8} Let $(R, \frak m)$ be the image of a Cohen-Macaulay local ring and $M$ a finitely generated $R$-module with the unmixed component $U_M(0)$. Then
$$pf(M) = p(M/U_M(0)).$$
\end{corollary}
Recall that an $R$-module $M$ is called {\it pseudo (generalized) Cohen-Macaulay} if $pf(M) = 0$ (resp. $pf(M) \le 0$). As a consequence of Corollary \ref{C3.8} we get a generalization of the main result of \cite{CN03}.
\begin{corollary} Let $(R, \frak m)$ be the image of a Cohen-Macaulay local ring and $M$ a finitely generated $R$-module with the unmixed component $U_M(0)$. Then
M is pseudo Cohen-Macaulay (resp. pseudo generalized Cohen-Macaulay) iff $M/U_M(0)$ if Cohen-Macaulay (resp. generalized Cohen-Macaulay).
\end{corollary}

It is natural to raise the following question.
\begin{question}\rm Let $M$ be an unmixed finitely generated $R$-module of dimension $d$ and $\underline{x} = x_1, ..., x_d$ a $C$-system of parameters of $M$. Is it true that that the function $J_{\underline{x}, M}(\underline{n})$ is a polynomial in $\underline{n}$ when $n_1, ..., n_d \gg 0$?
\end{question}
It should be noted that \cite[Theorem 4.5]{CM96} gives an affirmative answer for this question in the case $pf(M) \le 1$.

\section{The case $M$ admits a Macaulayfication}
\begin{definition}\rm Let $M$ be a finitely generated $R$-module of dimension $d$. We say that  $M$ admits a {\it Macaulayfication} $M'$
if we have an exact sequence
$$0\to {M }\to M' \to
{N
}\to 0,$$ where $M'$ is a finitely generated Cohen-Macaulay $R$-module and $\dim N\leq d-2$.

\end{definition}
\begin{remark}[see for example \cite{M07}, \cite{Sch82}] \rm Let $(R,\frak m)$ be a Noetherian complete local ring and $M$ a finitely generated $R$-module of dimension $d$. We recall that if $M$ is unmixed, the module  $D^d(D^d(M))$ (where  $D^d(M)$ is the Matlis dual of $M$)
satisfies the condition $S_{2}$ and we have an exact sequence :
$$0\to {M }\to {D^d(D^d(M)) }\to N
\to 0$$ with $\dim N\leq d-2$.
Moreover if there exist a finitely generated $R$-module $M'$ of dimension $d$, satisfying the
condition $S_{2}$ and an exact sequence :
$$0\to {M}\to {M'}\to {M'/M }\to 0$$ with
$\dim M'/M\leq d-2$, then $M' \cong D^d(D^d(M))$. That is, if $M$ is unmixed the Macaulayfication is unique up to isomorphism (if exist). In this is the case, $\mathrm{Supp}(M'/M)$ is the non Cohen-Macaulay locus of $M$.
\end{remark}
We can state the main result of this section.
\begin{theorem}\label{T4.3}
  Let $M$ be finitely generated $R$-module of dimension $d$. Suppose that $M$ has a Macaulayfication $M'$. Let $\underline{x} = x_1, ..., x_d$ be an arbitrary system of parameters of $M$. Set $N  = M'/M$, then
$$J_{\underline{x}, M}(\underline{n})  = \ell(N/(\underline{x}^{\underline{n}})N)$$
for all $d$-tuples $\underline{n} = (n_1, ..., n_d)$.
\end{theorem}
\begin{proof} For any system of parameters $\underline{y} = y_1, ..., y_d$, the short exact sequence
$$0 \to M \to M' \to N \to 0$$
induces the following commutative diagram with the last two columns exact
\[\divide\dgARROWLENGTH by 2
\begin{diagram}
\node{} \node{} \node{H^{d-1}(\underline{y}; N)} \arrow{e}\arrow{s} \node{H^{d-1}_{\frak m}(N) = 0}\arrow{s}\\
\node{0}\arrow{e}\node{(\underline{y})_M^{\lim}/(\underline{y})M} \arrow{e,t}{\beta}\arrow{s}\node{H^d(\underline{y};M)}\arrow{e}\arrow{s,l}{\alpha}
\node{H^d_{\frak m}(M)}\arrow{s}\\
\node{}\node{0}
\arrow{e}\node{H^d(\underline{y};M')}\arrow{e}\arrow{s} \node{H^d_{\frak m}(M')}\arrow{s}\\
\node{} \node{} \node{H^{d}(\underline{y}; N)} \arrow{e}\arrow{s} \node{H^{d}_{\frak m}(N) = 0}\\
\node{} \node{} \node{0.}
\end{diagram}
\]
Both the second and the third rows are exact by Remark \ref{R2.9}. Therefore we have $\alpha \circ \beta = 0$.  Thus we have the following commutative diagram
\[\divide\dgARROWLENGTH by 2
\begin{diagram}
\node{0}
\arrow{e,}\node{M/(\underline{y})_M^{\lim}}\arrow{e,t}{\pi}\arrow{s,l}{\overline{\alpha}}
\node{H^d_{\frak m}(M)}\arrow{s,l}{\sigma}\\
\node{0}
\arrow{e}\node{M'/(\underline{y})M'} \arrow{e,t}{\tau} \arrow{s}\node{H^d_{\frak m}(M')}\\
\node{} \node{N/(\underline{y})N)}\arrow{s}\\
\node{} \node{0}
\end{diagram}
\]
with the middle column is exact. Moreover we have both $\pi$ and $\tau$ are injective and $\sigma$ is bijective. Therefore $\tau \circ \overline{\alpha} = \sigma \circ \pi$ is injective and so is $\overline{\alpha}$. Hence we have the following short exact sequence
$$0 \to M/(\underline{y})_M^{\lim} \to M'/(\underline{y})M' \to N/(\underline{y})N \to 0.$$
Thus
$$\ell(M/(\underline{y})_M^{\lim}) = \ell(M'/(\underline{y})M') - \ell(N/(\underline{y})N).$$
Now for each $\underline{n} = (n_1, ..., n_d)$, applying the above assertion for the system of parameters $\underline{x}^{\underline{n}} = x_1^{n_1}, ..., x_d^{n_d}$ we have
  $$\ell(M/(\underline{x}^{\underline{n}})_M^{\lim}) = \ell(M'/(\underline{x}^{\underline{n}})M') - \ell(N/(\underline{x}^{\underline{n}})N).$$
  Since $M'$ is Cohen-Macaulay we have
  $$\ell(M'/(\underline{x}^{\underline{n}})M') = e(\underline{x}^{\underline{n}}; M') =e(\underline{x}^{\underline{n}}; M). $$
Therefore $J_{\underline{x}, M}(\underline{n})  = \ell(N/(\underline{x}^{\underline{n}})N)$ for all $d$-tuples $\underline{n} = (n_1, ..., n_d)$. The proof is complete.
\end{proof}
The length $\ell(N/(\underline{x}^{\underline{n}})N)$ is much easier to understand than the function $J_{\underline{x}, M}(\underline{n})$. In many cases we can see that it coincides with a polynomial or a finite number of polynomials for $\underline{n} \gg 0$. The following Corollary extends \cite[Lemma 2.4]{CMN03}.
\begin{corollary}\label{C4.4}  Let $(R ,{\frak m}) $ be a Cohen-Macaulay local ring of dimension  $d\geq 3$, $x_1, ..., x_d$  a system of parameters of $R$.
 Let $M = (x_1, ..., x_{d-v})$, $v \le d-2$. Then for the system of parameters $\underline{x} = x_1 + x_d, x_2, ..., x_d$ of $M$ we have
  $$J_{\underline{x}, M} (\underline{n}) = \ell(R/(x_1, ..., x_{d}))\,\, n_{d-v+1}...n_{d-1} \min \{ n_1, n_d\}$$
for all $n_1, ..., n_d \ge 1$. Therefore $J_{\underline{x}, M} (\underline{n})$ is not a polynomial.
\end{corollary}
\begin{proof} Since $\dim R/M \le d-2$, $R$ is a Macaulayfication of $M$. By Theorem \ref{T4.3} we have
$$J_{\underline{x}, M}(\underline{n})  = \ell(R/(x_1, ..., x_{d-v}, (x_1+x_d)^{n_1}, x_2^{n_2}, ..., x_d^{n_d})$$
for all $n_1, ..., n_d \ge 1$. Hence
\begin{eqnarray*}
  J_{\underline{x}, M}(\underline{n})  &=& \ell(R/(x_1, ..., x_{d-v}, x_{d-v+1}^{n_{d-v+1}}, ..., x_{d-1}^{n_{d-1}}, x_d^{\min \{n_1, \, n_d\}})\\
&=&  \ell(R/(x_1, ..., x_{d}))\,\, n_{d-v+1}...n_{d-1} \min \{ n_1, n_d\}
\end{eqnarray*}
for all $n_1, ..., n_d \ge 1$.
\end{proof}

The next result follows from Theorem \ref{T4.3} and Proposition \ref{P2.7}.
\begin{corollary}\label{C4.5}
 Let $M$ be a finitely generated $R$-module of dimension $d$. Suppose that $M$ has a Macaulayfication $M'$ with $\dim M'/M = t$. Let $\underline{x} = x_1, ..., x_d$ be any system of parameters of $M$ such that $x_1, ..., x_t$ forms a $dd$-sequence of $N  = M'/M$ and $x_{t+1}, ..., x_d \in \mathrm{Ann}N$. Then $J_{\underline{x}, M}(\underline{n})$ is a polynomial in $\underline{n}$ for all $n_1, ..., n_d \ge 1$. Moreover
 $$J_{\underline{x}, M}(\underline{n}) = n_1...n_t e(x_1, ..., x_t; N) +
\sum_{i=0}^{t-1} n_1...n_i e_i,$$ where
$e_i = e(x_1,...,x_i; 0:_{N/(x_{i+2},...,x_t)N}x_{i+1})$ and $e_0 =
\ell(0:_{N/(x_{2},...,x_t)N}x_{1})$.
\end{corollary}

\section{Relation with the Hilbert-Kunz function}
By considering all explicit examples, it can be expected  that  $J_{\underline{x}, M}(\underline{n})$  coincides with finitely
 many polynomials in $\underline{n}$ (cf. \cite{CMN03}, \cite{MN03}). As we will see this is not always the case. More precisely,
 we will give an example in characteristic $p$ such that the function $J_{\underline{x}, M}(\underline{n})$ can not be controlled by
finitely many polynomials. This question  is closely related to the
Hilbert-Kunz function.\\
Let $(A, \frak n)$ be a Noetherian local ring containing a field of positive characteristic $p$. Let $I$ be an ideal of $A$ and a prime power
$q = p^e$ we define $I^{[q]} = (f^q | q \in I)$ as the $e$-th Frobenius power of $I$. If $I$ is an $\frak n$-primary ideal we always have that $A/I^{[q]}$
 has finite length. So we have a function
$$f_{HK}(I): q \mapsto \ell (A/I^{[q]}),$$
called the Hilbert-Kunz function, which was  first studied by E. Kunz in \cite{K69}.
In \cite{Mo83}, P. Monsky proved that the limit
$$e_{HK}(I) = \lim_{q \to \infty} \frac{\ell (A/I^{[q]})}{q^{\dim A}}$$
exists as a real number; it is called the Hilbert-Kunz multiplicity of $I$, and
the Hilbert-Kunz multiplicity of $\frak n$ is also called the Hilbert-Kunz multiplicity
of $A$. It is natural to ask whether the Hilbert-Kunz multiplicity of an $\frak n$-primary ideal is always a rational number.
 There are many positive partial answers to this question. However, recently H. Brenner disproved this question by the following celebrate result.
\begin{theorem}[\cite{B14}, Theorem 8.3] There exists a Noetherian local domain whose Hilbert-Kunz
multiplicity is an irrational number.
\end{theorem}

We are ready to prove the main result of this section.

\begin{theorem}\label{T5.2} There exist a regular local ring $(R, \frak m)$ of dimension $d$ with $\frak m$ generates by a regular system of parameters $\underline{x} = (x_1, ..., x_d)$ and a finitely generated $R$-module $M$, $\dim M = d$ such that the function $J_{\underline{x}, M}(n) = n^d e(\underline{x}; M) - \ell(M/(\underline{x}^{[n]})^{\lim}_M)$ can not be represented by finitely many polynomials in $n$, where $\underline{x}^{[n]} = x_1^n, ..., x_d^n$.
\end{theorem}

\begin{proof} Let $(A, \frak n)$ be the ring of characteristic $p$ whose Hilbert-Kunz multiplicity is irrational as Brenner's result. Replacing $A$ by its completion, notice that  the Hilbert-Kunz multiplicity does not change, we can assume that $(A, \frak n)$ is complete. By the Cohen structure theorem we have that $A$ is the image of a regular local ring $(R, \frak m)$ of dimension $d$. Since $e_{HK}(A)$ is irrational we have $A$ is not regular and so $\dim R - \dim A \ge 1$. If $\dim R - \dim A = 1$ we replace $R$ by $R[X]_{(\frak m, X)R[X]}$. Henceforth we can assume that $\dim R - \dim A \ge 2$. Let the $R$-module $M$ be the kernel of the canonical map $R \to A$, we have $\dim M = d$. Choose a regular system of parameters $\underline{x} = x_1, ..., x_d$ generates $\frak m$. By Theorem \ref{T4.3} we have
  $$J_{\underline{x}, M}(n) = \ell (A/(\underline{x}^{[n]})A)$$
for all $n \ge 1$.  For all $i =  1, ..., d$ we denote by $a_i $ the image of $x_i$ in $A$. We have the sequence $\underline{a} = a_1, ..., a_d$ generates the maximal ideal $\frak n$ of $A$. Now we assume that there are only finitely many polynomials $P_1(n), ..., P_r(n)$ such that for each $n \ge 1$ we have $J_{\underline{x}, M}(n) = P_i(n)$ for some $i$ and find a contradiction. We consider the case $n$ is a prime power $q = p^e$ we have
$$J_{\underline{x}, M}(q) = \ell (A/\underline{a}^{[q]})  = \ell(A/\frak n^{[q]}).$$
Since there are infinitely many $q$, we must have a polynomial, says $P_1(n)$, such that
$$\ell(A/\frak n^{[q]}) = P_1(q)$$
for infinitely many $q = p^e$. It should be noted that if a polynomial takes integer values at infinitely many integer numbers, then all of its coefficients are rational. Thus the leading coefficient of $P_1(n)$ is a rational number and $\mathrm{deg}P_1(n) = \dim A$. So
$$e_{HK}(A) = \lim_{q \to \infty} \frac{\ell (A/\frak n^{[q]})}{q^{\dim A}} = \lim_{q \to \infty} \frac{P_1(q)}{q^{\dim A}}$$
is a rational number. It is a contradiction with our assumption about $A$. The proof is complete.

\end{proof}
For the next result we need the concept of the principle of {\it idealization}. Let $(R, \frak m)$ be a Noetherian local ring and $M$ a finitely generated $R$-module. We make the Cartesian product $R \times M$ into a commutative ring with respect to
component-wise addition and multiplication defined by $(r, m) \cdot (r', m') = (rr', rm' + r'm)$. We call this the idealization of $M$ (over $R$) and denote it by $R \ltimes M$. The idealization
$R \ltimes M$ is Noetherian local ring with identity $(1, 0)$, its maximal ideal is $\frak m \times M$ and
its Krull dimension is $\dim R$. If $\underline{x} = x_1, ..., x_d$ is a system of parameters of $R$ then $(\underline{x, 0}) = (x_1, 0), ..., (x_d, 0)$ is a system of parameters of the idealization $R \ltimes M$.
\begin{lemma}[\cite{CMN03}, Lemma 2.6]\label{L5.3} Let $\dim M = \dim R = d$ and $S = R \ltimes M$. Let $\underline{x} = x_1, ..., x_d$ is a system of parameters of $R$. Then we have
  $$\ell(S/(\underline{x, 0})^{\lim}_S) = \ell(R/(\underline{x})^{\lim}_R) + \ell(M/(\underline{x})^{\lim}_M).$$
\end{lemma}

Now we prove the last result of this paper.
\begin{corollary} There exists a Noetherian local ring $(S, \frak n)$ of dimension $d$ and a system of parameters $\underline{y} = y_1, ..., y_d$ such that the function $J_{\underline{y},S}(n)$ can not be represented by finitely many polynomials in $n$.
  \end{corollary}
\begin{proof} We choose $(R, \frak m)$ and $M$ as in Theorem \ref{T5.2}. Let $\underline{x} = x_1, ..., x_d$ be a regular system of parameters of $R$. Let $S = R \ltimes M$ and $\underline{y} = (x_1, 0), ..., (x_d, 0)$. We can check that $e(\underline{y}; S) = e(\underline{x}; R) + e(\underline{x}; M)$. Since $R$ is regular we have $(\underline{x}^{[n]})^{\lim}_R = (\underline{x}^{[n]})$ for all $n$. So
$$\ell(R/(\underline{x}^{[n]})^{\lim}_R) = \ell(R/(\underline{x}^{[n]}) = n^de(\underline{x}; R).$$
Combining with Lemma \ref{L5.3} we have
\begin{eqnarray*}
  J_{\underline{y},S}(n) &=& \ell(S/(\underline{y}^{[n]})^{\lim}_S) - n^d e(\underline{y}; S)\\
  &=& \big( \ell(R/(\underline{x}^{[n]})^{\lim}_R) + \ell(M/(\underline{x}^{[n]})^{\lim}_M) \big) - n^d \big(e(\underline{x}; R) + e(\underline{x}; M)  \big)\\
  &= &  \ell(M/(\underline{x}^{[n]})^{\lim}_M)  - n^de(\underline{x}; M)\\
  &=& J_{\underline{x}, M}(n).
\end{eqnarray*}
The assertion now follows from Theorem \ref{T5.2}. The proof is complete.
\end{proof}

\noindent
{\bf Acknowledgments:}
This paper was finished  during the  second author's visit at the Institute Fourier, Grenoble, France.
He would like to thank the Institute Fourier and LIA Formath Vietnam, CNRS, for their support and hospitality.

\end{document}